\documentclass[a4paper,12pt]{article}
\usepackage{setspace,url, graphicx, float, subcaption, caption,amsmath}
\usepackage[top=2.5cm, bottom=2.5cm, right=2.5cm, left=2.5cm]{geometry}
\usepackage[hidelinks]{hyperref}
\usepackage{fancyhdr}
\usepackage{tabu}
\usepackage{ amssymb }
\usepackage{amsthm}
\usepackage{imakeidx}
\usepackage{titletoc}
\usepackage{xcolor}

\newcommand{\Cay}{\mbox{Cay}}
\newcommand{\Aut}{\mbox{Aut}}

\newtheorem{exm}{Example}[subsection]
\newtheorem{thm}{Theorem}[subsection]
\newtheorem{thms}{Theorem}[section]
\newtheorem{dfn}{Definition}[subsection]
\newtheorem{cors}{Corollary}[section]

\newtheorem{lems}{Lemma}[section]

\begin{document}
\title{Using Schur Rings to Produce GRRs for Dihedral Groups.}
\author{Jonathan Ebejer, Josef Lauri \\University of Malta}

\maketitle

\vspace{4mm}


\begin{abstract}
In this paper we shall be looking at several results relating Schur rings to sufficient conditions for a graph to be a graphical regular representation (GRR) of  finite group, and then applying these specifically in the case of certain subfamilies of dihedral groups. Numerical methods are given for constructing trivalent GRRs for these dihedral groups very quickly.
\end{abstract}

\section{Foundations}

\subsection{Introduction}

\begin{dfn}[Graphical\ Regular\ Representations]
	Let $\Gamma$ be a finite group and let $G$ be a graph. If $\Aut(G) \equiv \Gamma$ acts regularly on $V(G)$, then we say that $G$ is a \textbf{graphical regular representation} (GRR) of $\Gamma$.
\end{dfn}


The question asking which finite groups have at least one GRR has been completely solved \cite{Go, He, Im, Im2, ImWa, ImWa2, NoWa, NoWa2}. Specifically:

\begin{enumerate}
	\item The only abelian groups which have a GRR are $\mathbb{Z}^n_2$ for $n\geq5$.
	
	\item Except for generalised dicyclic groups and a finite number of known groups, all non-abelian groups have a GRR.
\end{enumerate}

However, it remains a challenge to discern whether a group known to have GRRs has GRRs with specific properties, such as being trivalent. It is well-known, by means of Frucht's Theorem \cite{BeWi}, that every group is isomorphic to the automorphism group of some trivalent graph but it is not guaranteed that this graph is a GRR.  Furthermore, it is not known how one can easily construct such GRRs when they do exist. In Section 2 we will be using the foundations laid out in this section to produce infinite families of trivalent GRRs for certain dihedral groups. Coxeter, Frucht, and Powers had given a thorough treatment of trivalent GRRs in their book \emph{Zero-Symmetric
Graphs: Trivalent Graphical Representations of Groups} \cite{Cox} which shows how interesting such GRRs can be.

\begin{dfn}[Cayley graphs]
The \textbf{Cayley graph} of a group $\Gamma$ is the graph, denoted by $\Cay(\Gamma, S)$, whose vertex-set is $\Gamma$ and two vertices $u$ and $v$ are adjacent if $v=us$ where $s\in S$ and $S$ is a subset of $\Gamma$ such that $1\not\in S$, $S$ generates $\Gamma$, and $S^{-1}=S$. We call the set $S$ the \textbf{connecting set} of the Cayley graph.

In dealing with Schur rings we might encounter Cayley graphs where the condition $S^{-1}=S$ is not satisfied. We refer to such graphs as Cayley digraphs. 
\end{dfn}

It is an easy and well-known corollary of Sabidussi's Theorem \cite{Sa} that if a graph $G$ is a GRR of a group $\Gamma$, then it isomorphic to a Cayley graph $\Cay(\Gamma,S)$ of $\Gamma$.  Therefore, if $G$ is cubic then the connecting set must either consist of three elements all of order 2 or else it must consist of one element $a$ and its inverse $a^{-1}\not=a$ and one other element $b$ of order 2. 

\subsection{Schur rings}

We now present an introduction to Schur rings which is just sufficient for our purposes.



\begin{dfn}[Group Rings]
Let $\Gamma$ be a finite group. Then the \textbf{group ring} $\mathbb{Z}[\Gamma]$ consists of all formal linear combinations 
	\[ \sum_{\gamma\in \Gamma} z_\gamma\cdot \gamma \]
	Where each $z_\gamma$ is in $\mathbb{Z}$.
	Addition and multiplication in $\mathbb{Z}[\Gamma]$ is carried out in the natural way.
	If $B=\{b_1,b_2,\ldots,b_p\}\subseteq \Gamma$, then the element 
	\[ b_1 + b_2 +\ldots+b_p\]
	of $\mathbb{Z}[\Gamma]$ is denoted by
	$\overline{B}$ and it is said to be  a \textbf{simple quantity} of $\mathbb{Z}[\Gamma]$.
\end{dfn}

\begin{dfn}[Schur Rings]
A subring $\mathcal{S}$ of the group ring $\mathbb{Z}[\Gamma]$ is called a Schur ring or an $\mathcal{S}$-ring over $\Gamma$, of rank $r$, if the following conditions hold:
\begin{itemize}
	\item $\mathcal{S}$ is closed under addition and multiplication including multiplication by elements of  $\mathbb{Z}$ from the left (i.e. $\mathcal{S}$ is a $\mathbb{Z}[\Gamma]$-module).
	\item Simple quantities $\overline{B}_{0},\overline{B}_{1},...,\overline{B}_{r-1}$ exist in $\mathcal{S}$ such that every element ${C} \in \mathcal{S}$ has a unique representation:
	\[{C}=\sum_{i=0}^{r-1}\alpha_{i}\overline{B}_{i}\] where the $\alpha_i$ are integers.
	\item $\overline{B}_{0}=\overline{\{1\}}$.
	\item $\displaystyle \sum_{i=0}^{r-1}\overline{B}_{i}=\overline{\Gamma}$, that is, $\lbrace B_{0},B_{1},...,B_{r-1}\rbrace$ is a partition of $\Gamma$.
	\item For every $i\in \lbrace 0,1,2,...,r-1 \rbrace$ there exists a $j \in \lbrace 0,1,2,...,r-1 \rbrace$ such that $\overline{B}_{j}=\overline{B}^{-1}_{i} (=\overline{\lbrace b^{-1} : b \in B_{i} \rbrace )}$.
\end{itemize}
\end{dfn}

We call the set of simple quantities $\overline{B}_{0},\overline{B}_{1},\ldots,\overline{B}_{r-1}$the \textbf{basis} of the Schur ring and we denote it by ${\mathcal B}[{\mathcal S}]$. Each simple quantity $\overline{B}_i$ of the basis is referred to as  a \textbf{basic element} of the Schur ring. Sometimes we need to refer to the set $B_i$ which we call a \textbf{basic set}. If $\gamma \in \Gamma$ and $\overline{\{\gamma\}}$ is a basic element of $\mathcal{S}$, then we say the $\gamma$ is \textbf{isolated} or a \textbf{singleton} in ${\mathcal B}[{\mathcal S}]$. If the basis of a Schur ring is completely made up of singletons then we say that the Schur ring is \textbf{trivial}.

We say a Schur ring is \textbf{larger} than another if its basic sets form a partition which is finer than the partition formed by the other's basic sets. The ``largest'' Schur ring of a group $\Gamma$ is therefore the Schur ring whose basic sets form the finest possible partition of $\Gamma$, that is, the basic sets are all the singleton elements of $\Gamma$. Similarly, we say a Schur ring is \textbf{smaller} than another if its basic sets form a partition which is coarser than the partition formed by the other's basic sets. The ``smallest'' Schur ring is therefore the one which gives the coarsest partition, with basic sets $\{1\}$ and $\Gamma - \{1\}$.


The example below  deals with a special kind of Schur ring known as a $\textbf{cyclotomic}$ Schur ring  \cite{De, Wi}.


\begin{exm}
	Let $\Delta$ be a group of automorphisms of the group $\Gamma$. Then the orbits of $\Gamma$ under the action of $\Delta$ form the basic sets of a Schur ring over $\Gamma$.
		
	\bigskip	
	In particular, the conjugacy classes of $\Gamma$ form the basic sets of a Schur ring over $\Gamma$. The basic elements in this case would here be the well-known ``class sums'' which are very important in linear representations of groups.
\end{exm}
	
Because of closure under multiplication, the product of two linear combinations of $\overline{B}_{0},\overline{B}_{1},...,\overline{B}_{r-1}$ must also be a linear combination of these simple quantities. Therefore we can make the following definition:

\vspace{2mm}

\begin{dfn}[Structure Constants]
Let $\overline{B}_i$ and $\overline{B}_j$ be two basis elements of an r-rank Schur ring $\mathcal{S}$. For all values $i, j, k \in [r]$ there exist non-negative integers $\beta_{i,j}^{k}$ called \textbf{structure constants}, such that
\[\overline{B}_{i}\cdot\overline{B}_{j}=\sum_{k=1}^{r} \beta_{i,j}^{k}\overline{B}_{k}\]
\end{dfn}

We shall soon see that the structure constants have a very nice graph theoretic interpretation.


\begin{exm}\label{exm:basiccayley}
	Let $\Gamma$ be the cyclic group $\langle \gamma: \gamma^8=1\rangle$. The following simple quantities are the basic sets of a Schur ring of $\Gamma$:
	\[ \{\{1\}, \{\gamma^1,\gamma^5 \}, \{\gamma^3, \gamma^7\}, \{\gamma^2,\gamma^6 \}, \{\gamma^4\} \}. \]
	 
	One can verify that, if we let $B_0=\{1\}, B_1=\{\gamma^1,\gamma^5\}, B_2=\{\gamma^3,\gamma^7\}, B_3=\{\gamma^2,\gamma^6\}$, and $B_4=\{\gamma^4\}$, then,
	$\overline{B_2}\cdot\overline{B_4} = \overline{B_2}$, therefore $\beta_{2,4}^2=1$ while all the other $\beta^k_{2,4}=0$.

\end{exm}

\subsection{Graph theoretic interpretation}

\begin{dfn}[Basic Cayley Graphs]
Let $\mathcal{S}$ be a Schur ring of a group $\Gamma$. We can construct a Cayley (di)graph $\Cay(\Gamma,B_i)$ for each of the basic sets $B_i$ of $\mathcal{S}$. These are called the \textbf{basic Cayley graphs} associated with the Schur ring.
\end{dfn}

The basic Cayley graphs associated with the Schur ring in the previous example are as shown in Figure \ref{fig:basiccayley}.  

	\begin{figure} 
	\includegraphics[width=\linewidth]{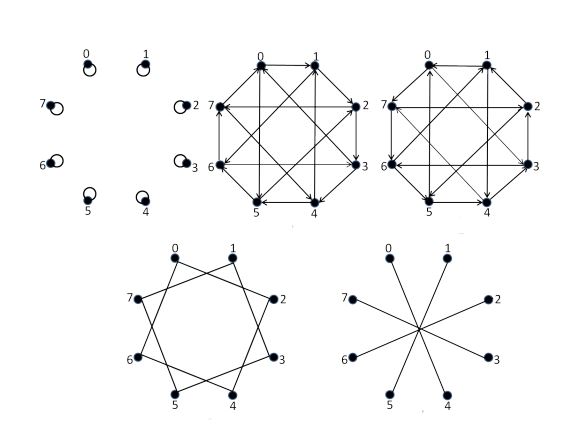}		
	\caption{The basic Cayley graphs associated with the Schur ring in Example \ref{exm:basiccayley}} \label{fig:basiccayley} 
	\end{figure}
		
We can now give an interpretation to the structure constants $\beta^k_{i,j}$: pick \emph{any} edge (or arc in the case of a Cayley digraph) $ab$ in $\Cay(G,B_k)$ and count how many walks there are from $a$ to $b$ by first passing through an edge/arc in $\Cay(G,B_i)$ followed by an edge/arc in $\Cay(G, B_j)$.  This is the value of $\beta^k_{i,j}$.

One can				
verify here that $\beta^1_{2,3}=2$ and $\beta^4_{2,1}=2$, for example.

\subsection{The Colour Graph}

Basic Cayley graphs such as the ones in Figure \ref{fig:basiccayley} can be superimposed on the vertex-set consisting of the elements of $\Gamma$ with each one having its edges/arcs coloured a different colour. This gives a colouring of the arc-set of the complete graph where the convention is to represent two arcs $(u,v), (v,u)$ which happen to have the same colour represented as an edge of that colour.
Figure \ref{fig:colourgraph} shows the previous set of basic Cayley graphs without the loops.

\begin{figure}
\begin{center}

\includegraphics[width= 200pt]{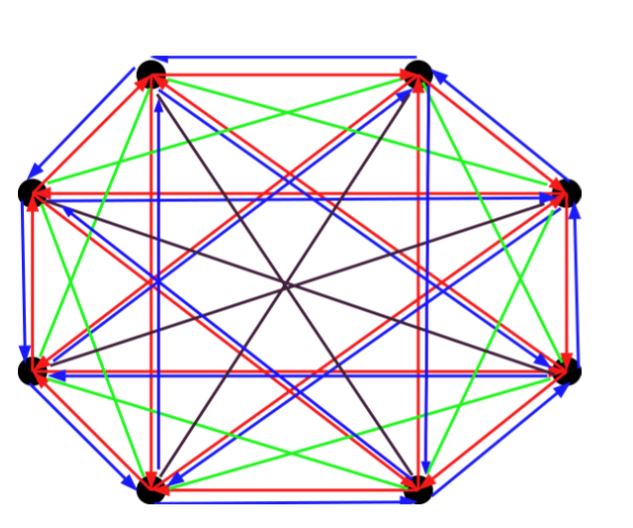}
\caption{The colour graph for the Schur ring in Example \ref{exm:basiccayley}}
\label{fig:colourgraph}
\end{center}
\end{figure}

As we did above, we can verify here from the colour graph that $\beta^{red}_{blue,green}=2=\beta^{black}_{blue,red}$, for example. 
	
The adjacency matrices of the basic Cayley (di)graphs form what is called a \textbf{coherent configuration}.

\subsection{Automorphism groups of Schur rings}
	
\begin{dfn}[Automorphisms of Schur Rings]
The automorphism group of a Schur ring is defined to be the intersection of all automorphism groups of the basic Cayley graphs of the Schur ring. For any Schur ring $\mathcal{S}$ we will denote its automorphism group as $Aut(\mathcal{S})$.
\end{dfn}


The smallest (coarsest) Schur ring has the largest automorphism group, that is, the symmetric group $S_n$, where $n$ is the order of the group. The largest Schur ring has the smallest automorphism group, that is,  the regular action of the group on itself.
	
\medskip	
		Let $C$ be a subset of the group $\Gamma$. Then $\langle\langle C \rangle\rangle$ is the smallest (coarsest) Schur ring of $\Gamma$ containing $\overline{C}$.

Theorem \ref{thm:aut<<C>>}\cite{MuPo} and Theorem \ref{thm:trivial<<C>>} \cite{Wei} will be very important going forward.



\begin{thm}\label{thm:aut<<C>>}
	$\Aut(\Cay(\Gamma,C))=\Aut(\langle\langle C \rangle\rangle )$.
\end{thm} 


\begin{thm}\label{thm:trivial<<C>>}
If $\langle \langle C \rangle \rangle$ is trivial then $Cay(\Gamma,C)$ is a GRR of $\Gamma$. 
\end{thm}

\subsection{Combining Results}
	
Our recipe for finding GRRs for a group $\Gamma$ is simple. First, find $C$ such that $\langle\langle C \rangle\rangle$ is trivial. Theorem 1.5.2 then gives us that, since $\langle \langle C \rangle\rangle $ is trivial, $\Aut(\Cay(\Gamma,C)) = \Gamma$ and so $\Cay(\Gamma,C)$ is a GRR of $\Gamma$.

Our work in Section 2 combines these two steps into one by identifying properties of a connecting set of a dihedral group which imply that the Cayley graph of that set must be a GRR. By selecting a connecting set consisting of three elements of order 2, we will also be producing trivalent graphs. One last result about Schur rings will be our main computational tool going forward. A proof is given in \cite{KlRuRuTi}.

\begin{thm}[Schur-Wielandt Principle]
Let $r=\sum z_\gamma \gamma$ be an element of a Schur ring $\mathcal{S}$ of a group $\Gamma$. Then, for any integer $k$, the sum $\sum_{z_\gamma=k} \gamma$ is also in $\mathcal{S}$.
\end{thm}


So, this result is saying, for example, that if  $a+2b +c + 3d + 2f$ is in $\mathcal{S}$, then so are $a+c$, $b+f$, and $d$.

We finish off this section with a concrete example of how we use these results to show that a particular Cayley graph is a GRR using some simple {\tt GAP} code. \cite{GAP4} used to carry out the algebraic calculations.
	
\begin{exm}
Let $D_7$ be the dihedral group 
	\[ \langle a, b: a^2=b^7=abab=1\rangle.\] 
	
We show that $\Cay(D_7, C)$ where $C = \{a, ab, ab^3, b, b^6\}$ is a GRR of $G$.

	\begin{tt} \small
		\begin{verbatim} 
gap> d7:=DihedralGroup(14);;
gap> Zd7:=GroupRing(Integers, d7);;
gap> a:=Zd7.1;
(1)*f1
gap> b:=Zd7.2;
(1)*f2

gap> a^2;
(1)*<identity> of ...
gap> b^7;
(1)*<identity> of ...

gap> (a+a*b+a*b^3+b+b^6)^2;
(5)*<identity> of ...+
(4)*f1*f2^2+
(2)*f1+(2)*f1*f2+(2)*f2^2+(2)*f1*f2^4+(2)*f2^5+(2)*f1*f2^6+
(1)*f2^3++(1)*f2+(1)*f2^4+(1)*f2^6+
		\end{verbatim}
	\end{tt}

By computing $\overline{C}^2$ in $\mathbb{Z}[\Gamma]$ and using the Schur-Wielandt Principle, we obtain that 
$x= ab^2$ and $y= a + ab +b^2+ ab^4+ b^5 +ab^6$ are in $\langle\langle C\rangle\rangle$. 

Therefore, $xy =a+b^2+b^4+ab^4+b^5+b^6$ is in $\langle\langle C\rangle\rangle$.

But $(xy)^{-1} = a+b^5+b^3+ab^4+ b^2+ b$.
and
basic elements are disjoint, therefore either $b$ or $b+b^3$ is a basic element of  $\langle\langle C \rangle\rangle$. But by squaring the latter, we see that $b^4$ is a basic element of the Schur ring, hence so is $b$. Multiplying $x$ by $b$ repeatedly it follows that so is $a$. Therefore $\mathcal{S}$ is the finest Schur ring on $D_7$ with all singleton sets as basic sets.
	
Therefore $\Aut(\Cay(D_7,C)) = \Aut(\langle\langle C \rangle\rangle) = D_7$.  
\end{exm}

We finish this section with the following simple observation.

\vspace{2mm}

\textbf{Observation 1.} Let $\overline{A}$ be an element of the Schur ring $\mathcal{S}$. Then some partition of $A$ (which could be $A$ itself) gives a set of basic elements in $\mathcal{B}[\mathcal{S}]$


\section{Constructing cubic GRRs for dihedral groups}

At its core, our work in Section 2 is conceptually identical to the final example in the previous section, but rather than using specific values for the powers of group elements, we generalise the example by using parametric variables.


This will be our main result:


\begin{thms}\label{thm:main}
Let $n$ be an odd integer greater than 5 and let $r$, $s$, and $t$ be integers less than $n$ such that the difference of any two of them is relatively prime to $n$. If $3r-2s=t \bmod n$, then $\Cay(D_n, \{ab^r, ab^s, ab^t\})$ is a GRR of $D_n$.
\end{thms}


Which gives us the tidy corollary:


\begin{cors} \label{cor:tidy}
Let $n$ be an odd integer with its smallest prime factor being $p$ greater than 5 and let $r$, $s$, and $t$ be distinct integers less than $p$. If $3r-2s=t \bmod n$ then $\Cay(D_n, \{ab^r, ab^s, ab^t\})$ is a GRR of $D_n$.
\end{cors}


This will allow us to easily make trivalent GRRs for several dihedral groups just by selecting appropriate values for $r$, $s$ and $t$. It will be especially useful for those dihedral groups of order equal to twice a prime greater than 5 (dihedral groups $D_p$), where all we would need would be $r$, $s$, and $t$ such that $3r - 2s = t \bmod p$.

To prove Theorem 2.1 however, we will need the help of the following two lemmas:

\begin{lems} \label{lem1}
Let $\gamma$ be an element of $\Gamma$ such that $\gamma$ has odd order $n$ and let $r$ and $t$ be two integers such that $r-t$ is relatively prime to $n$. If $\gamma^r + \gamma^t$ is an element of some $\mathcal{B}[\mathcal{S}_\Gamma]$ then  $\gamma^s + \gamma^{-s}$ is also an element of that same $\mathcal{B}[\mathcal{S}_\Gamma]$ for all integers $s$ relatively prime to $n$.
\end{lems}

\begin{proof}
First we observe that since $\mathcal{S}_\Gamma$ is closed then $(\gamma^r + \gamma^t)^n$ is in $\mathcal{S}_{\Gamma}$ whenever $(\gamma^r + \gamma^t)$ is in $\mathcal{S}_{\Gamma}$. We can use binomial expansion to get:

$(\gamma^r+\gamma^ t)^n = \displaystyle\sum_{i=0}^{n}[\binom{n}{i}\gamma^{ri}\gamma^{ t(n-i)}]$ which can be re-written as $\displaystyle\sum_{i=0}^{n}[\binom{n}{i}\gamma^{(r- t)i)}]\gamma^{tn}$. However, $\gamma$ is of order $n$ and so $\gamma^{tn} \equiv 1$. Therefore, the sum can once more be rewritten as $\displaystyle\sum_{i=0}^{n}[\binom{n}{i}\gamma^{(r- t)i)}]$.

\medskip
Since $r- t$ is relatively prime to the order of $\gamma$ then the summation $\displaystyle\sum_{i=0}^{n}[\binom{n}{i}\gamma^{(r- t)i)}]$ must include exactly $n$  distinct terms. In other words every power of $\gamma$ appears in that sum once and only once. Therefore, the coefficient of every $\gamma^{(r-t)i}$ is exactly $\binom{n}{i}$. Moreover, we observe that $\binom{n}{i} = \binom{n}{n-i}$ and so we can re-write this summation as $\displaystyle\sum_{i = 0}^{\lfloor\frac{n}{2}\rfloor}[\binom{n}{i}(\gamma^{(r- t)i}+\gamma^{(t-r)i})]$. We use the floor function here as $n$ may be odd.

\medskip
This means that for every $i \in [0, \frac{n}{2}]$ either $\gamma^{(r-t)i}$ occurs as a singleton in $\mathcal{B}[\mathcal{S}_\Gamma]$ or $\gamma^{(r-t)i}+\gamma^{(t-r)i}$ occurs in $\mathcal{B}[\mathcal{S}_\Gamma]$ due to the Schur-Wielandt principle.

However, in the cases where $i$ is relatively prime to $n$, $(r-t)i$ would also be relatively prime to $n$, and so $\gamma^{(r-t)i}$ can generate every power of $\gamma$. This would mean that should $\gamma^{(r-t)i}$ be a singleton element of the basis, then every power of $\gamma$ would be a singleton element of the basis also. However, $\gamma^r + \gamma^t$ is an element of $\mathcal{B}[\mathcal{S}_\Gamma]$, and so this would imply a contradiction. Therefore for every $i \in [0, \frac{n}{2}]$ which is relatively prime to $n$, we have $\gamma^{(r-t)i}+\gamma^{(t-r)i} \in \mathcal{B}[\mathcal{S}_\Gamma]$.

Let us recapitulate that at this point we have shown that if $n$ is the order of some $\gamma \in \Gamma$, $r$ and $t$ are two integers such that $(r-t)$ are relatively prime to $n$ and $(\gamma^r + \gamma^t)$ is in $\mathcal{B}[\mathcal{S}_\Gamma]$ then $\displaystyle\sum_{i = 0}^{\lfloor \frac{n}{2} \rfloor} [\binom{n}{i}(\gamma^{(r- t)i}+\gamma^{(t-r)i})]$ is in $\mathcal{S}_{\Gamma}$ and that those pairs of summands of the form $\gamma^{(r- t)i}+\gamma^{(t-r)i}$ appear as a basis element in $\mathcal{B}[\mathcal{S}_\Gamma]$ when $i$ is relatively prime to $n$. All that is left to show is that every number less than $n$ which is also relatively prime to it appears in the expression $\displaystyle\sum_{i = 0;\ i\nmid n}^{\lfloor \frac{n}{2} \rfloor} [\binom{n}{i}(\gamma^{(r- t)i}+\gamma^{(t-r)i})]$ as a power of $\gamma$ (or, equivalently, a value of $(r-t)i$) and this will complete the proof.

We re-write  $\displaystyle\sum_{i = 0 \ i\nmid n}^{\lfloor \frac{n}{2} \rfloor}[\binom{n}{i}(\gamma^{(r- t)i}+\gamma^{(t-r)i})]$ as $\displaystyle\sum_{j \in S}[\binom{n}{j} \gamma^{j}]$ where $S$ is the set $\{(r-t)i$ mod $n : i \in [n];\ i \nmid n\}$.

However, we now note that the set $S$ is in fact equivalent to the set of integers less than $n$ which also divide $n$. This is because the modulo multiplication of $[n]$ by any number relatively prime to $n$ can be represented as a permutation of a set unto itself, which is a bijection. Moreover the product of two numbers both relatively prime to $n$ is again relatively prime to $n$, which means the subset of $[n]$ which consists only of integers relatively prime to $n$ is fixed under such a bijection.

Therefore:

$\displaystyle\sum_{j \in S}[\binom{n}{j} \gamma^{j}] \equiv \displaystyle\sum_{j \in [n];\ j \nmid n}[\binom{n}{i} \gamma^j]$

Obviously every number less than $n$ which is less relatively prime to it appears as a power of $\gamma$ on the right hand side above, meaning that those powers of $\gamma$ also appear on the left hand side above.

This means that every $\gamma^{(r-t)i}$ such that $(r-t)i$ is relatively prime to $n$ appears in the summation $\displaystyle\sum_{i = 0}^{\lfloor\frac{n}{2}\rfloor}[\binom{n}{i}(\gamma^{(r- t)i}+\gamma^{(t-r)i})]$ and as we stated above, this completes the proof.

\end{proof}

\begin{lems}

\label{lem2}
Let $n$ be an odd integer and let $a$ and $b$ be the generators of the group $D_n$ with orders 2 and $n$ respectively. Also let $r$, $s$ and $t$ be unique integers less than or equal to $n$. The Schur ring $\langle\langle ab^r+ab^s+ab^t\rangle\rangle $ must be trivial if the following are true:

\begin{enumerate}

\item $ab^r+ab^s+ab^t \not \in \mathcal{B}[\langle\langle ab^r+ab^s+ab^t\rangle\rangle ]$
\item The absolute value of the difference between any two of $r$, $s$ and $t$ is relatively prime to $n$.
\item The sum of any two of $r$, $s$ and $t$ is not equal to twice the third variable taken modulo $n$.

\end{enumerate}
\end{lems}



\begin{proof}
We know from (1) that $ab^r+ab^s+ab^t \not \in \mathcal{B}[\langle\langle ab^r+ab^s+ab^t\rangle\rangle ]$. But $ab^r+ab^s+ab^t$ is in $\mathcal{B}[\langle\langle ab^r+ab^s+ab^t\rangle\rangle ]$. This therefore leaves only two possibilities:

\begin{itemize}
\item $ab^{r}, ab^{s}, ab^{t}$ are all isolated in $\mathcal{B}[\langle\langle ab^r+ab^s+ab^t\rangle\rangle ]$, or,
\item Without loss of generality, $ab^{r}+ab^{t}$ is in $\mathcal{B}[\langle\langle ab^r+ab^s+ab^t\rangle\rangle ]$ while $ab^{s}$ is isolated.
\end{itemize}

The first possibility implies that $\langle\langle ab^r+ab^s+ab^t\rangle\rangle$ is trivial and so if this is true, we are done. Let us turn our attention to the second possibility.

\vspace{2mm}
This gives us that $(ab^{r}+ab^{t})ab^s$ is in $\langle\langle ab^r+ab^s+ab^t\rangle\rangle$. By expanding the brackets we get that $b^{s-r}+b^{s-t}$ is in $\langle\langle ab^r+ab^s+ab^t\rangle\rangle$. If $b^{s-r}$ and $b^{s-t}$ occur as singletons in the basis then it would imply that the whole Schur ring is trivial since $ab^s$ is already known to be in the basis and $\{ b^{s-r}, b^{s-t}, ab^s\}$ generates the whole group $D_n$. Let us assume then that $b^{s-r}+b^{s-t}$ is in $\mathcal{B}[\langle\langle ab^r+ab^s+ab^t\rangle\rangle]$, and mark this assumption as \textbf{\emph{Asm 1}}.

\medskip 
From the second part of our sufficient condition we also know that $t-r$ and $r-s$ are both relatively prime to $n$. Now $b^{s-r}+b^{s-t}$ is in $\mathcal{B}[\langle\langle ab^r+ab^s+ab^t\rangle\rangle ]$ and the difference between the two powers is $t-r$ which we know is relatively prime to $n$. Lemma \ref{lem1} therefore gives us that $b^{r-s} + b^{s-r}$ must also be in  $\mathcal{B}[\langle\langle ab^r+ab^s+ab^t\rangle\rangle ]$ since $r-s$ is also relatively prime to $n$. Both the elements $b^{s-r}+b^{s-t}$ and $b^{r-s} + b^{s-r}$ have $b^{s-r}$ within them, however, and under the assumption \textbf{\emph{Asm 1}} these should be basis elements. This is only possible if $b^{s-t} \equiv b^{r-s}$ and that implies that $r+t \equiv 2s$ mod $n$, which contradicts the third part of our sufficient condition.

\medskip 
Hence it must be the case that our assumption \textbf{\emph{Asm 1}} is false and so $\langle\langle ab^r+ab^s+ab^t\rangle\rangle$ must be trivial.
\end{proof}


We can now focus on proving Theorem \ref{thm:main}:

\vspace{4mm}


 \begin{proof}

We will prove Theorem \ref{thm:main} by showing that $ab^r+ab^s+ab^t$ cannot be an element of $\mathcal{B}[\langle \langle ab^r+ab^s+ab^t\rangle \rangle]$ and then use Lemma \ref{lem2} to meet the sufficient conditions of Theorem \ref{thm:trivial<<C>>} and so bring about the result.

We substitute $t$ for $3r -2s$ at this point for simplicity's sake. We will begin by assuming that $ab^r+ab^s+ab^{3r-2s}$ is in $\mathcal{B}[\langle \langle ab^r+ab^s+ab^{3r-2s}\rangle \rangle]$ and then show how this must imply a contradiction. We label this assumption as \textbf{\emph{(Asm 2)}} for further reference.

So let us begin by considering the following statement, which we label as \textbf{\emph{(i)}} for further reference:


\[(ab^r+ab^s+ab^{3r-2s})^2 = 3(1)+ b^{r-s}+b^{2(r-s)}+b^{3(r-s)}+b^{-(r-s)}+b^{-2(r-s)}+b^{-3(r-s)}\]

Assuming that \textbf{\emph{(Asm 2)}} is true then by Observation 1 it must be possible to express \textbf{\emph{(i)}} as the sum of elements of $\mathcal{B}[\langle \langle ab^r+ab^s+ab^{3r-2s}\rangle \rangle]$. However, we shall show that doing so implies a contradiction which means that \textbf{\emph{(Asm 2)}} must be false.

We begin by noting that the group elements on the right hand side of \textbf{\emph{(i)}} must either be isolated in the basis of the Schur ring or be grouped with other elements also on the right hand side of \textbf{\emph{(i)}}.

For the sake of brevity let us denote the sum of the elements on the right hand side as $B_6$. We note now that $(B_6)^2$ is equal to


\begin{align*}
6(1) + 2B_6 + b^{2(r-s)}+b^{-2(r-s)}+b^{6(r-s)}+b^{-6(r-s)} + \\
2(b^{r-s}+b^{s-r}+b^{5(r-s)}+b^{-5(r-s)})+3(b^{4(r-s)}+b^{-4(r-s)}).
\end{align*}

We notice that the total coefficient of $b^{r-s}$ on the right hand side is 4 (remember it occurs in $B_6$), and so, by the Schur Wielandt principle the only elements which $b^{r-s}$ can be grouped with in the Schur ring's basis must also have a coefficient of 4 in the expansion of $B_6^2$. However, with the exception of $b^{s-r}$, achieving this requires that $\pm 4(r-s), \pm 5(r-s)$ or $\pm 6(r-s)$ is equal to some element in $B_6$. This implies $x(r-s)$ = $y(r-s)$ for $x= \pm 1, \pm 2$ or $\pm 3$ and $ y = \pm 4, \pm 5$ or $\pm 6$. Regardless of the values of $x$ and $y$ we get $r-s =0$ implies $r =s$ since we are working in modulo $n$ where $n$ is definitely larger than 6. Therefore $b^{r-s}$ can only possibly be grouped with $b^{s-r}$ or be an isolated element in the basis.

Let us first consider the possibility that $b^{r-s}$ is an isolated element. If so, then $b^{3(s-r)}$ must also be an isolated element in the basis. Since our assumption is that $ab^r+ab^s+ab^{3r-2s}$ is in $\mathcal{B}[\langle \langle ab^r+ab^s+ab^{3r-2s}\rangle \rangle]$ we can consider the product of $ab^r+ab^s+ab^{3r-2s}$ and $b^{3(r-s)}$, which must be in $\langle \langle ab^r+ab^s+ab^{3r-2s}\rangle \rangle$.


\[(ab^r+ab^s+ab^{3r-2s})b^{3(r-s)} \equiv ab^{4r-3s}+ab^{3r-2s}+ab^{6r-5s}.\]


Note how the term $ab^{3r-2s}$ appears on the right hand side above. Since we are assuming that $ab^r+ab^s+ab^{3r-2s}$ is in $\mathcal{B}[\langle \langle ab^r+ab^s+ab^{3r-2s}\rangle \rangle]$, $ab^r$ and $ab^s$ must appear too. There are only two ways this is possible:

\begin{itemize}
\item $4r-3s = s$ and $6r-5s = r$, or,
\item $4r-3s = r$ and $6r-5s = s$.
\end{itemize} 

All the equations above, however, imply that $r=s$ which is a contradiction. And so $b^{r-s}$ cannot be an isolated element in the basis without contradicting \textbf{\emph{(Asm 2)}}. Therefore, the only way to avoid contradicting \textbf{\emph{(Asm 2)}} is for $b^{r-s}+b^{s-r}$ to be an element in the basis.

This gives us that $(ab^r + ab^s + ab^{3r-2s})(b^{r-s}+b^{s-r})$ is in the Schur ring. This expression expands to $2ab^{2r-s}+ab^{2s-r}+ab^{4r-3s}+ ab^s +ab^r$. Notably the terms $ab^s, ab^r$ appear here and due to \textbf{\emph{(Asm 2)}} this means that $ab^{3r-2s}$ must appear also. There are three ways this could happen:

\begin{itemize}
\item $2r-s=3r-2s$ gives us $r = s$, or,
\item $4r-3s=3r-2s$ gives us $r = s$, or,
\item $2s-r=3r-2s$ gives us $r = s$.
\end{itemize}

But all these ways imply $r=s$ which is again a contradiction. This rules out the possibility that $b^{r-s}+b^{s-r}$ is an element in the basis.

Therefore, \textbf{\emph{(Asm 2)}} must be false and  $ab^r+ab^s+ab^{3r-2s}$ is not in $\mathcal{B}[\langle \langle ab^r+ab^s+ab^{3r-2s}\rangle \rangle]$ for if it were, then it would be a basis element which cannot be expressed as a linear sum of basis elements when squared, which contradicts the definition of a Schur ring. Lemma 2.2 gives us that this implies that $\langle\langle ab^r+ab^s+ab^t\rangle\rangle $ is trivial and so $\Aut(\langle \langle ab^r+ab^s+ab^t \rangle \rangle) \equiv D_n$ meaning that $\Cay(D_n, \{ ab^r,ab^s,ab^t \})$ is a GRR of $D_n$ by Theorem \ref{thm:trivial<<C>>}.

 \end{proof}


Corollary \ref{cor:tidy} follows immediately by observing that $r$, $s$ and $t$, being less than $p$, immediately implies that both they and their differences are relatively prime to $n$. 

We also note the following theorem, although we invite the reader to read its proof in a companion paper to this one which is available on arXiv \cite{Eb&La2} as that proof is very similar to the proof above.


\begin{thms}
	Let $n$ be an odd integer greater than 5 and let $r$, $s$, and $t$ be integers less than $n$ such that the difference of any two of them is relatively prime to $n$. If $3r+s=4t$ mod $n$, then $\Cay(D_n, \{ab^r, ab^s, ab^t\})$ is a GRR of $D_n$.
\end{thms}


Just as with the prior theorem, Theorem 2.2 has a tidy corollary for prime numbers:


\begin{cors} \label{cor:tidy2}
Let $n$ be an odd integer with its smallest prime factor being $p$ greater than 5 and let $r$, $s$, and $t$ be distinct integers less than $p$. If $3r+s=4t$ mod $n$, then $\Cay(D_n, \{ab^r, ab^s, ab^t\})$ is a GRR of $D_n$.
\end{cors}

\vspace{2mm}

In the next section we shall give a few examples in which the above results are applied. 

\pagebreak

\section{Examples of GRR constructions}

We will be focusing on applying Corollary \ref{cor:tidy} to when $n$ itself is prime, giving us GRRs of $D_p$, as this will provide clear and concise examples.




\subsection{A simple application}

First let us start with $p =11$. To satisfy the equation $3r-2s=t$ we will take $r=3$, $s=4$ and $t=1$. Therefore $Cay(D_{11}, \{ab^3, ab^4, ab\})$ is a GRR of $D_{11}$. We present the Cayley colour graph for this connecting set in Figure \ref{fig:graph2} so that the reader can see which element of the connecting set gave rise to which edge: $ab$ is red, $ab^3$ is blue and $ab^4$ is green. These edge-colourings hold true irrespective of which vertex is the identity element of $D_{11}$ since every automorphism of a GRR only maps edges of the same colour to each other \cite{ABCFHMST}.



\begin{figure}[H]
\begin{center}
  \includegraphics[width=\linewidth]{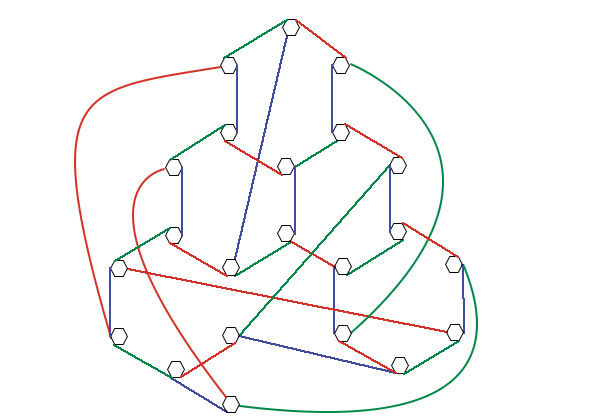}
  \caption{$Cay(D_{11}, \{ab, ab^3, ab^4\})$, a GRR of $D_{11}$.}
  \label{fig:graph2}
  \end{center}
\end{figure}

\subsection{The inheritability of connecting sets}

In the previous example we used the values $\{ 3, 4, 1 \}$ for the parameters $\{r, s, t\}$ respectively to construct a GRR for $D_{11}$. However, with a very simple theorem we can show that these values for $\{ r, s ,t \}$ can be used to construct GRRs for an infinite number of dihedral groups $D_p$ where $p$ is prime.

To do this, we prove the following theorem. Note that all arithmetic in the proof is only modular arithmetic if its explicitly said to be so, and all order relations are to be understood as an ordering on the positive integers:

\begin{thm}Let $p_2 > p_1$ be prime. Suppose that, for a given $\alpha, \beta , \gamma$ there exist integers $r,s,t$ such that $\alpha r - \beta s \equiv \gamma t  \ mod \ p_1.$ Then, if $\alpha r - \beta s \in [0, p_1]$ then  $\alpha r - \beta s \equiv \gamma t \ mod \ p_2$ also holds. \label{thm:321}
\end{thm}

\vspace{2mm}

\begin{proof}

$\alpha r - \beta s \equiv \gamma t  \ mod \ p_1$ implies that $\alpha r - \beta s = \gamma t + xp_1$ for some integer $x$. We observe that $\alpha r - \beta s \in [0, p_1]$ means that $x = 0$.

Now consider the expression  $\alpha r - \beta s  \ mod \ p_2$. Since $p_2 > p_1$ we can re-write this expression as $(\gamma t + xp_1) \ mod \ p_2$. When $x = 0$ this simplifies to $\gamma t \ mod \ p_2$.

Therefore if $\alpha r - \beta s \in [0, p_1]$ then $\alpha r - \beta s \equiv \gamma t \ mod \ p_2$ holds.

\end{proof}

\vspace{2mm}

This theorem has a useful corollary:

\begin{cors} \label{cor:inherit}
Let $p_2 > p_1$ be prime. Let $D_{p_1}$ be a dihedral group such that $D_{p_1} \equiv \langle a,b : a_1^2 \equiv b_1^{p_1} \equiv a_1b_1a_1b_1 \equiv 1  \rangle$. Let $r,s,t$ be integers which satisfy $3r - 2s \equiv t \ mod \ p_1$ and therefore let us construct a connecting $C_1$ set for a GRR of $D_{p_1}$ where $C_1 \equiv \{a_1b_1^r, a_1b_1^s, a_1b_1^t\}$. Then those same integers $r,s,t$ can be used to construct a connecting set $C_2$ which admits a GRR for $D_{p_2}$ where $D_{p_2} \equiv \langle a,b : a_2^2 \equiv b_2^{p_2} \equiv a_2b_2a_2b_2 \equiv 1  \rangle$ and $C_2 \equiv \{a_2b_2^r, a_2b_2^s, a_2b_2^t\}$.
\end{cors}

With this, any solutions for $r,s,t$ which we find when trying to form a GRR for some dihedral group $D_{p_1}$ with $p_1$ prime can be re-used to construct connecting sets for an infinite number of other dihedral groups, specifically those dihedral groups $D_{p_2}$ where $p_2 > p_1$ and prime.
 
Figure \ref{fig:graph3} is a GRR of $D_{13}$ which we construct using the same values we found when constructing a GRR for $D_{11}$ in the previous example.

\begin{figure}[H]
  \includegraphics[width=\linewidth]{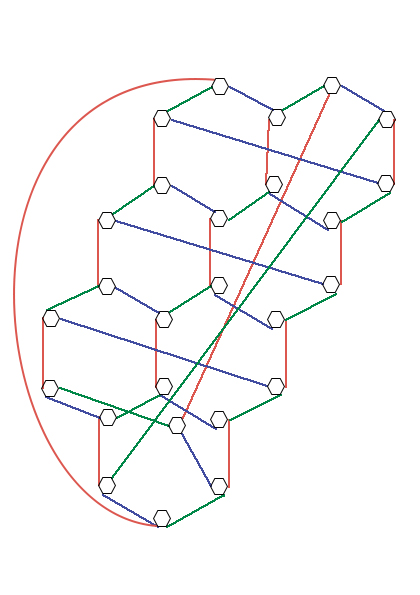}
  \caption{$Cay(D_{13}, \{ab, ab^3, ab^4\})$, a GRR of $D_{13}$.}
  \label{fig:graph3}
\end{figure}


\subsection{Some further solutions}
Below is a table of connecting sets for $D_p$ for a few prime values of $p$, found using the equation $3r-2s = t$ mod $p$ (see Corollary \ref{cor:tidy}). While constructing these connecting sets, the condition $\alpha r - \beta s \in [0, p]$  was respected and so we know from Corollary \ref{cor:inherit} that all the sets in the table below are in fact connecting sets for GRRs of $D_p$ not just for the prime number given in its row but for all prime numbers larger than it too.

\vspace{2mm}

For brevity we present only the prime numbers up to 100.

\vspace{2mm}

\begin{center}
\begin{tabular}{r|| c | c | c | c }
$p$&$r$&$s$&$t$&Connecting Set($ \{ab^r, ab^s, ab^t\}$)\\
\hline
7&2&3&0&$ \{ab^2, ab^3, a\}$\\
11&3&4&1&$ \{ab^3, ab^4, ab\}$\\
13&4&6&0&$ \{ab^4, ab^6, a\}$\\
17&5&7&1&$ \{ab^5, ab^7, ab\}$\\
19&6&9&0&$ \{ab^6, ab^9, a\}$\\
23&7&10&1&$ \{ab^7, ab^{10}, ab\}$\\
29&9&13&1&$ \{ab^9, ab^{13}, ab\}$\\
31&10&15&0&$\{ab^{10}, ab^{15}, a\}$\\
37&12&18&0&$\{ab^{12}, ab^{18}, a\}$\\
41&13&19&1&$\{ab^{13}, ab^{19}, ab\}$\\
43&14&21&0&$\{ab^{14}, ab^{21}, a\}$\\
47&15&22&1&$\{ab^{15}, ab^{22}, ab\}$\\
53&17&25&1&$\{ab^{17}, ab^{25}, ab\}$\\
59&19&28&1&$\{ab^{19}, ab^{28}, ab\}$\\
61&20&30&0&$\{ab^{20}, ab^{30}, a\}$\\
67&22&33&0&$\{ab^{22}, ab^{38}, a\}$\\
71&23&34&1&$\{ab^{23}, ab^{34}, ab\}$\\
73&24&36&0&$\{ab^{24}, ab^{36}, a\}$\\
79&26&39&0&$\{ab^{26}, ab^{39}, a\}$\\
83&27&40&1&$\{ab^{27}, ab^{40}, ab\}$\\
89&29&43&1&$\{ab^{29}, ab^{43}, ab\}$\\
97&32&48&0&$\{ab^{32}, ab^{48}, a\}$\\

\end{tabular}
\end{center}

\vspace{2mm}

\pagebreak

\section{Conclusion}

By making use of a few theorems and lemmas about Schur rings we have shown easy ways to construct trivalent GRRs for dihedral groups $D_n$ where $n$ is an odd number. It is the opinion of the authors that these results are evidence that even a basic study of Schur rings can be fruitful not only in the search for GRRs but also in other areas of algebraic graph theory.

We are now in a position where we can produce two GRRs for each member of an infinite family of dihedral groups within moments. For example, $\{a, ab^2, ab^3 \}$ and $\{a, ab^3, ab^4 \}$ are connecting sets for GRRs of all $D_n$ where $n$ is prime and larger than 11, by using Corollaries \ref{cor:tidy}, \ref{cor:tidy2} and \ref{cor:inherit}. 

This can be the groundwork for several further lines of inquiry. For example, are the methods we have outlined sufficient to find all possible trivalent GRRs for the relevant dihedral groups? If no, can we perhaps find a similar approach to account for the remaining ones? If yes, can these methods be used as groundwork for enumeration theorems?

Moreover, the methods we have described here can be expanded to find GRRs for dihedral groups with a valency greater than 3 and we can also slightly relax the conditions on the parameter $n$. This is all to be discussed another time, however.

These lines of inquiry we have outlined may yet yield further useful results.

\bibliographystyle{plain}

\end{document}